\documentclass[11pt]{article}
\usepackage{amsmath,amssymb,amsthm,comment}
\usepackage{amsfonts}
\usepackage{color}
\usepackage{epsfig}
\newtheorem{theo}{Theorem}[section]
\newtheorem{lemma}[theo]{Lemma}
\newtheorem{cor}[theo]{Corollary}
\newtheorem{remark}[theo]{Remark}
\newtheorem{prop}[theo]{Proposition}

\makeatletter
    
    \@addtoreset{equation}{section}
  \makeatother
\newcommand{\io}{\int_0^1}
\newcommand{\R}{\mathbb{R}}
\newcommand{\N}{\mathbb{N}}

\newcommand{\dx}{\partial_x}
\newcommand{\dt}{\partial_t}
\newcommand{\abs}{\\[5pt]}
\newcommand{\tm}{T_{\mathrm{max}}}
\begin{document}
\title{Global existence in the 1D quasilinear parabolic-elliptic chemotaxis system with critical nonlinearity}
\author{
Tomasz Cie\'slak\footnote{cieslak@impan.pl}\\
{\small Institute of Mathematics,}\\
{\small Polish Academy of Sciences,}\\
{\small Warsaw, 00-656, Poland}
\and
Kentarou Fujie\footnote{fujie@rs.tus.ac.jp}\\
{\small Tokyo University of Science,}\\
{\small Tokyo, 162-0861, Japan}
}
\date{\small\today}
\maketitle
\begin{abstract}
The paper should be viewed as complement of an earlier result in \cite{CL_DCDS}. In the paper just mentioned it is shown that 1d case of a quasilinear parabolic-elliptic Keller-Segel system is very special. Namely, unlike in higher dimensions, there is no critical nonlinearity. Indeed, for the nonlinear diffusion of the form $1/u$ all the solutions, independently on the magnitude of initial mass, stay bounded. However, the argument presented in \cite{CL_DCDS} deals with the J\"{a}ger-Luckhaus type system. And is very sensitive to this restriction. Namely, the change of variables introduced in \cite{CL_DCDS}, being a main step of the method, works only for the J\"{a}ger-Luckhaus modification. It does not seem to be applicable in the usual version of the parabolic-elliptic Keller-Segel system. The present paper fulfils this gap and deals with the case of the usual parabolic-elliptic version. To handle it we establish a new Lyapunov-like functional (it is related to what was done in \cite{CL_DCDS}), which leads to global existence of the initial-boundary value problem for any initial mass.
\abs
\end{abstract}
\newpage
\section{Introduction}\label{section_introduction}
We consider the following type of PDE system
\begin{align*}
	\begin{cases}
	u_t=\nabla \cdot \left( a(u) \nabla u-  u \nabla v \right)
	&\mathrm{in}\ (0,\infty)\times \Omega, \\[1mm]
	0=\Delta v -v+u
	&\mathrm{in}\ (0,\infty)\times \Omega, 
	\end{cases}
\end{align*}
with a given smooth function $a$ and $\Omega \subset \R^n$ ($n \in \N$).
The above system was introduced to describe an aggregation of cells (\cite{KS}). 
There are several known Lyapunov functionals associated to the above system (see for instance \cite{NSY}).
The information brought by Lyapunov functional is often a starting point in the studies of behavior of solutions  of Keller-Segel system, see for instance \cite{BBTW}.

\bigskip
{\bf Background}\quad
In the higher dimensional case ($n\geq 2$), it is known that the critical diffusion related to the
Keller-Segel system, meaning such that stronger diffusion yields global-in-time bounded solutions, while weaker
allows initial data (independently on the magnitude of mass) leading to finite-time blowup, is 
\begin{equation}\label{eq}
a(u) = u^{1-\frac{2}{n}}.
\end{equation}
Indeed, such a result in higher dimensions can be found in \cite{S_DIE} in the case of the domain being the whole $\R^n$. For bounded domains we refer to \cite{CW}. Moreover in the case of a diffusion behaving asymptotically like \eqref{eq}, the critical mass phenomenon holds. Namely there exists a 
threshold value of initial mass distinguishing between bounded and exploding solutions.
Again, we can refer the reader to \cite{S_ADE}, \cite{BCL} in the whole space case and to \cite{CL_CRAS} (finite-time blowup of radially symmetric solutions) and \cite{Nasri} (global existence for masses small enough) for bounded domains. For a deeper introduction to the topic we refer the reader to the Introduction of \cite{TCPhL4}. 
In the case of $a(u)=u^p$ or $a(u)=(1+u)^p$ the identified critical value is $p=1-\frac{2}{n}$. For this parameter one finds critical mass phenomenon in dimensions $n\geq 2$. 
Based on these results, $a(u)=1/u$ or $a(u)=(1+u)^{-1}$ is a candidate for 
the critical nonlinearity in one-dimensional setting. 
Indeed, in the case of simplified J\"{a}ger-Luckhaus system $a(u) = (1+u)^p$ 
with $p>-1$ implies global-in-time boundedness of solutions; 
whereas the case $p<-1$ allows solutions blowing up in finite time for initial data with arbitrary mass, see \cite{CW}. 
Surprisingly, in \cite{CL_DCDS} it is shown that for any initial mass solutions to the simplified J\"{a}ger-Luckhaus system with critical diffusion remain bounded. 
Let us mention here that such result is still lacking its counterpart in the fully parabolic case. So far it is only known that for small initial mass solutions exist globally and remain bounded, see \cite{BCM-R}. On the other hand we remark that also in the case of nonlocal diffusion in 1d there is no critical nonlinearity (see \cite{bg}). 
The aim of the present note is to extend the result of \cite{CL_DCDS} to the case of usual parabolic-elliptic 1d Keller-Segel system, meaning to show that the J\"{a}ger-Luckhaus modification is not essential.  

\bigskip
{\bf Motivation}\quad
Let us be more precise here and tell the reader what is the J\"{a}ger-Luckhaus modification exactly.
Its 1d version was studied both in \cite{CW, CL_DCDS}. 
\begin{align}\label{P_motivation}
	\begin{cases}
	\dt u =\dx \left( a(u) \dx u -  u \dx v \right)
	&\mathrm{in}\ (0,T)\times(0,1), \\[1mm]
	 0 =\dx^2 v -M +u, \;\;\int_0^1 vdx=0, 
	&\mathrm{in}\ (0,T)\times(0,1),  \\[1mm]
	 \dx A(u)(t,0)=\dx A(u)(t,1)=\dx v(t,0)=\dx v(t,1)=0
	&t\in(0,T),  \\[1mm]
	u(0,\cdot)=u_0\geq 0
	&\mathrm{in}\ (0,1),
	\end{cases}
\end{align} 
where $a\in C^1(0,\infty)$ is a positive function,  
the function $A$ is a primitive of $a$ and $M= \io u_0$. 
In \cite{CL_DCDS}, a change of variables in the the above system is introduced. It simplifies it to a single parabolic equation and it is proved that this parabolic equation has a Lyapunov functional. The information coming from the obtained Lyapunov functional is rich enough to establish global existence of solutions of \eqref{P_motivation}. 
However, the change of variables is very sensitive to the fact that the lower equation in \eqref{P_motivation}
has its exact form and cannot be pursuit for the usual Keller-Segel version. 

However, expressing the Lyapunov functional appearing in the reduced 1d problem in \cite{CL_DCDS} in the original 
variables we get the following Lyapunov functional associated to the full system \eqref{P_motivation}
\begin{eqnarray*}
\dfrac{d}{dt} \mathcal{F}_0(u(t)) + \mathcal{D}_0(u(t),v(t))
=0,
\end{eqnarray*}
where
\begin{eqnarray*}
\mathcal{F}_0(u(t)) 
&:=&\dfrac{1}{2} \io \dfrac{(a(u))^2}{u}|\dx u|^2 
+\io \left(-\int_1^u a(r)\,dr + M \int_1^u \frac{a(r)}{r}\,dr  \right)u, \\[2mm]
\mathcal{D}_0(u(t),v(t)) 
&:=& \io ua(u)\left| \dx \left(\dfrac{a(u)}{u}\dx u - \dx v \right) \right|^2.
\end{eqnarray*}
The above is different from the classical Lyapunov functional. 
We prove in the present note that a similar functional occurs also in the 
1d chemotaxis system without J\"{a}ger-Luckhaus modification. It carries enough information to be useful in the discussion of global existence. 
%
%
%

\bigskip
{\bf Main result}\quad
We consider the following problem
\begin{align}\label{P1}
	\begin{cases}
	\dt u =\dx \left( a(u) \dx u -  u \dx v \right)
	&\mathrm{in}\ (0,T)\times(0,1), \\[1mm]
	 0 =\dx^2 v -v+u
	&\mathrm{in}\ (0,T)\times(0,1),  \\[1mm]
	 \dx A(u)(t,0)=\dx A(u)(t,1)=\dx v(t,0)=\dx v(t,1)=0
	&t\in(0,T),  \\[1mm]
	u(0,\cdot)=u_0\geq 0
	&\mathrm{in}\ (0,1),
	\end{cases}
\end{align} 
where $a$ is a positive function, either $a\in C^1(0,\infty)\cap C[0,\infty)$ 
or singularity of $a$ at $0$ is admitted. 
The function $A$ is an indefinite integral of $a$ and 
$u_0 \in C^1[0,1]$ such that 
\begin{equation}\label{dodatnie}
u_0 \geq 0 \qquad \mbox{in } (0,1)
\end{equation}
if $a\in C[0,\infty)$ or
$u_0 \in C^1[0,1]$ such that 
\begin{equation}\label{nieujemne}
u_0 \geq m_0>0 \qquad \mbox{in } (0,1)
\end{equation}
with some $m_0>0$ if the singularity $a(0)$ is admitted.  

We construct the following functional satisfying
\begin{eqnarray*}
\dfrac{d}{dt} \mathcal{F}(u(t)) + \mathcal{D}(u(t),v(t))
=\io \dfrac{ua(u)v^2}{4},
\end{eqnarray*}
where
\begin{eqnarray*}
\mathcal{F}(u(t)) 
&:=&\frac{1}{2} \io \dfrac{(a(u))^2}{u}|\dx u|^2
- \io u\int_1^u a(r)\,dr ,\\[2mm]
\mathcal{D}(u(t),v(t)) 
&:=& \io ua(u) 
\left|  \dx \left(\dfrac{a(u)}{u}\dx u \right) - \dx^2 v  + \dfrac{v}{2} \right|^2.
\end{eqnarray*}
Our main result reads as follows.  
\begin{theo}\label{main_theorem}
Let $a(u)=1/u$ or $a(u)=\frac{1}{1+u}$. 
Then the problem \eqref{P1} has a unique classical positive solution, which exists globally in time. 
\end{theo}
\begin{remark}
Comparing with the study of the problem \eqref{P_motivation}, 
though the formal difference between \eqref{P_motivation} and \eqref{P1} seems small, 
we had to find another way of dealing with the problem. We cannot apply the change of variables method, used in \cite{CL_DCDS}, to the problem \eqref{P1}. Also, the functional that we construct is similar,
but slightly different than the one in \cite{CL_DCDS}. Actually our new functional is not even a Lyapunov functional. 
But we still can control its growth in time. 
\end{remark}
\medskip

{\bf Plan of the paper.}\quad This paper is organized as follows. 
Section 2 is devoted to the key identity in this paper 
and a collection of useful facts. 
After constructing a new Lyapunov-like functional in Section 3, 
we derive regularity estimates and give a proof of Theorem \ref{main_theorem} (Section 4). 
%
%
%
%
%
%
%
\section{Preliminaries}\label{section_preliminaries}
The next lemma is the key ingredient in this paper. 
\begin{lemma}\label{Key}
Let $\phi \in C^3((0,1))$. 
Then the following identity holds:
\begin{align*}
\phi \dx \mathcal{M}(\phi)
=\dx \left( \phi a(\phi) \dx \left(\dfrac{a(\phi)}{\phi}\dx \phi \right) \right),
\end{align*}
where 
\begin{align*}
\mathcal{M}(\phi) 
:= \dfrac{a(\phi)a'(\phi)}{\phi} |\dx \phi|^2
-\dfrac{(a(\phi))^2}{2\phi^2} |\dx \phi|^2
+\dfrac{(a(\phi))^2}{\phi}\dx^2 \phi.
\end{align*}
\end{lemma}
\begin{proof}
By a straightforward calculation we will show the identity.
The left-hand side of the identity is calculated as
\begin{align*}
&\, 
\phi \dx \mathcal{M}(\phi) \\
&=
\phi \left(
\dfrac{(a'(\phi))^2}{\phi} |\dx \phi|^2 \dx \phi
+
\dfrac{a(\phi)a''(\phi)}{\phi} |\dx \phi|^2 \dx \phi
-
\dfrac{a(\phi)a'(\phi)}{\phi^2} |\dx \phi|^2 \dx \phi
+
\dfrac{a(\phi)a'(\phi)}{\phi} \dx (|\dx \phi|^2)
  \right)\\
&\quad +\phi \left( 
-\dfrac{a(\phi)a'(\phi)}{\phi^2} |\dx \phi|^2 \dx \phi
+
\dfrac{(a(\phi))^2}{\phi^3} |\dx \phi|^2 \dx \phi
-\dfrac{(a(\phi))^2}{2\phi^2} \dx (|\dx \phi|^2)
\right)\\
&\quad +\phi \left(
\dfrac{2a(\phi)a'(\phi)}{\phi}  \dx^2 \phi \dx \phi
-\dfrac{(a(\phi))^2}{\phi^2} \dx^2 \phi \dx \phi
+\dfrac{(a(\phi))^2}{\phi} \dx^3 \phi
 \right)\\
 &=
(a'(\phi))^2 |\dx \phi|^2 \dx \phi
+
a(\phi)a''(\phi) |\dx \phi|^2 \dx \phi
-
\dfrac{2a(\phi)a'(\phi)}{\phi} |\dx \phi|^2 \dx \phi
+
a(\phi)a'(\phi) \dx (|\dx \phi|^2)\\
&\quad 
+
\dfrac{(a(\phi))^2}{\phi^2} |\dx \phi|^2 \dx \phi
-\dfrac{(a(\phi))^2}{2\phi} \dx (|\dx \phi|^2)\\
&\quad 
+2a(\phi)a'(\phi) \dx^2 \phi \dx \phi
-\dfrac{(a(\phi))^2}{\phi} \dx \phi \dx^2 \phi
+(a(\phi))^2 \dx^3 \phi.
\end{align*}
Since 
\begin{align*}
\dx \left(a(\phi)a'(\phi)|\dx \phi|^2 \right)
&=(a'(\phi))^2 |\dx \phi|^2 \dx \phi 
+a(\phi)a''(\phi) |\dx \phi|^2 \dx \phi
+a(\phi)a'(\phi) \dx (|\dx \phi|^2),\\
\dx \left((a(\phi))^2 \dx^2 \phi \right)
&=2a(\phi)a'(\phi)  \dx^2 \phi \dx \phi
+(a(\phi))^2 \dx^3 \phi,
\end{align*}
we compute that 
\begin{align*}
\phi \dx \mathcal{M}(\phi)
&=
\dx \left(a(\phi)a'(\phi)|\dx \phi|^2 \right)
+
\dx \left((a(\phi))^2 \dx^2 \phi \right)\\
&\quad 
-\dfrac{2a(\phi)a'(\phi)}{\phi} |\dx \phi|^2 \dx \phi
+
\dfrac{(a(\phi))^2}{\phi^2} |\dx \phi|^2 \dx \phi
-\dfrac{(a(\phi))^2}{2\phi} \dx (|\dx \phi|^2)
-\dfrac{(a(\phi))^2}{\phi} \dx^2 \phi \dx \phi.
\end{align*}
Due to the identity
\begin{align}\label{formula}
\dx \phi \dx^2 \phi = \dfrac{1}{2}\dx \left(|\dx \phi|^2 \right),
\end{align}
we arrive at
\begin{align*}
\phi \dx \mathcal{M}(\phi)
&=
\dx \left(a(\phi)a'(\phi)|\dx \phi|^2 \right)
+
\dx \left((a(\phi))^2 \dx^2 \phi \right)\\
&\quad -\dfrac{2a(\phi)a'(\phi)}{\phi} |\dx \phi|^2 \dx \phi
+
\dfrac{(a(\phi))^2}{\phi^2} |\dx \phi|^2 \dx \phi
-\dfrac{(a(\phi))^2}{\phi} \dx (|\dx \phi|^2)\\
&= \dx \left(
a(\phi)a'(\phi)|\dx \phi|^2 
+(a(\phi))^2 \dx^2 \phi
-\dfrac{(a(\phi))^2}{\phi} |\dx \phi|^2 \right)\\
&=
\dx \left( \phi a(\phi) \dx \left( \dfrac{a(\phi)}{\phi} \dx \phi \right)\right).
\end{align*}
The proof is completed.
\end{proof}
\begin{remark}
Since we used \eqref{formula}, the above calculation is valid only in one-dimensional setting.
\end{remark}
The following inequality is obtained in \cite{BHN, NSY, BCM-R}.
\begin{lemma}\label{BHN_inequality}
For $w\in H^1(0,1)$ and any $\delta>0$ there exists $C_\delta>0$ 
such that
\begin{align*}
\|w\|^4_{L^4(0,1)}
 \leq \delta \|w\|^2_{H^1(0,1)} \io |w \log w| 
 + C_\delta \|w\|_{L^1(0,1)}.  
\end{align*}
\end{lemma}
The next lemma can be proven the same way as in \cite{CW, CL_DCDS}.
\begin{lemma}\label{prop_local_existence}
 For $u_0$ satisfying \eqref{dodatnie} or \eqref{nieujemne} (depending whether $a$ is bounded or singular at $0$), there exist $T_{\mathrm{max}}\leq \infty$
 {\rm(}depending only on ${\|{u_0}\|}_{L^{\infty}(\Omega)}${\rm)} 
 and exactly one pair $(u,v)$ of positive functions
 \begin{align*}
(u,v) \in C([0,\tm)\times [0,1]; \R^2) \cap
C^{1,2}((0,\tm)\times [0,1]; \R^2)
   \end{align*}
 that solves {\rm(\ref{P1})} in the classical sense. 
 Also, the solution $(u,v)$ satisfies the mass identities
   \begin{equation*}
     \int_{\Omega}u(x,t)\,dx=\int_{\Omega}u_0(x)\,dx
     \quad \text{for\ all}\ t \in (0,T_{\mathrm{max}}).
   \end{equation*}
In addition, if $T_{\mathrm{max}} < \infty$, then
   \begin{equation*}
     \limsup_{t \nearrow T_{\mathrm{max}}}
     {\|{u(t)}\|}_{L^{\infty}(\Omega)}
     = \infty.
   \end{equation*}   
\end{lemma}
In virtue of the conservation of the total mass $\|u\|_{L^1(\Omega)}$, we can get the following regularity estimates, which is provided in \cite{BS}.
\begin{lemma}\label{prop_ellipticReg}
There exists some constant $M=M(\|u_0\|_{L^1(0,1)},p)>0$ such that
\begin{align*}
\sup_{t\in[0,\tm)}\|v\|_{L^{p}(0,1)} \leq M,
\end{align*}
where $p \in [1,\infty)$.
\end{lemma}
%
%
%
%
%
%
%
%
%
%
%
\section{New Lyapunov-like functional}\label{section_new_lyapunov}
In this section we construct a functional 
associated to the problem \eqref{P1}. It is not a classical Lyapunov functional, since
it does not decrease along the trajectories. However, we can control its growth which
allows us to infer the required estimates.
 
In order to construct the Lyapunov-like functional we will prepare two lemmas. 
We first invoke Lemma \ref{Key} to derive the next lemma. 
\begin{lemma}\label{lemma_Lyap1}
Let $(u,v)$ be a solution of \eqref{P1} in $(0,T)\times (0,1)$. 
Then the following identity holds
\begin{align*}
\dfrac{d}{dt} \left(\frac{1}{2} \io \dfrac{(a(u))^2}{u}|\dx u|^2 \right)
+ \io u a(u) \left| \dx \left(\dfrac{a(u)}{u}\dx u \right) \right|^2
= \io u a(u) \dx^2 v \cdot \dx \left(\dfrac{a(u)}{u}\dx u \right).
\end{align*}
\end{lemma}
\begin{proof}
Multiplying the first equation of \eqref{P1} by $\mathcal{M}(u)$ and integrating over $(0,1)$, we have that
\begin{align*}
\io \dt u \mathcal{M}(u)
&=\io \dx \left( a(u) \dx u -  u \dx v \right) \mathcal{M}(u)\\
&=\io \dx \left(u \left(\dfrac{a(u)}{u} \dx u -   \dx v\right) \right) \mathcal{M}(u).
\end{align*}
By the integration by parts and Lemma \ref{Key}, it follows that
\begin{align}\label{Est:diffusion1_P1}
\notag
\io \dt u \mathcal{M}(u)
&=
- \io \left(\dfrac{a(u)}{u} \dx u -   \dx v\right) \cdot u\dx\mathcal{M}(u)\\
\notag
&=- \io \left(\dfrac{a(u)}{u} \dx u -   \dx v\right) \cdot \dx \left(u a(u) \dx \left(\dfrac{a(u)}{u}\dx u \right) \right)\\
\notag
&=\io \dx \left(\dfrac{a(u)}{u} \dx u -   \dx v\right) \cdot  \left(u a(u) \dx \left(\dfrac{a(u)}{u}\dx u \right) \right)\\
&= \io u a(u) \left| \dx \left(\dfrac{a(u)}{u}\dx u \right) \right|^2
- \io u a(u) \dx^2 v\cdot \dx \left(\dfrac{a(u)}{u}\dx u \right).
\end{align}
On the other hand, we infer that
\begin{align*}
\dfrac{d}{dt} \left( \frac{1}{2} \io \dfrac{(a(u))^2}{u}|\dx u|^2 \right)
&=
\io \dfrac{2a(u)a'(u)u - (a(u))^2}{2u^2}  |\dx u|^2 \dt u
+
\io \dfrac{(a(u))^2}{u} \dx u \dx \dt  u\\
&=
\io \dfrac{2a(u)a'(u)u- (a(u))^2}{2u^2}  |\dx u|^2 \dt u
-\io \dx \left(\dfrac{(a(u))^2}{u} \dx u \right) \dt u.
\end{align*}
Since 
\begin{align*}
-\io \dx \left(\dfrac{(a(u))^2}{u} \dx u \right) \dt u
=
-\io \dfrac{2a(u)a'(u)u-(a(u))^2}{u^2}  |\dx u|^2 \dt u
-
\io \dfrac{(a(u))^2}{u}  \dx^2 u  \dt u,
\end{align*}
we deduce that
\begin{align}\label{Est:diffusion2_P1}
\dfrac{d}{dt} \left( \frac{1}{2} \io \dfrac{(a(u))^2}{u}|\dx u|^2 \right)
= - \io \dt u \mathcal{M}(u).
\end{align}
Combining \eqref{Est:diffusion1_P1} and \eqref{Est:diffusion2_P1}, we complete the proof.
\end{proof}
\begin{lemma}\label{Est1_P1}
Let $(u,v)$ be a solution of \eqref{P1} in $(0,T)\times (0,1)$. 
Then the following identity holds
\begin{align*}
\dfrac{d}{dt} \left(
\io u\int_1^u a(r)\,dr 
 \right)
=\io \dx \left(\dfrac{a(u)}{u} \dx u \right) u^2a(u)
+\io \dx \left( u^2 a(u) \right) \cdot \dx v
\end{align*}
\end{lemma}
\begin{proof}
Testing the first equation of \eqref{P1} by $\int_1^u a(r)\,dr + u a(u)$ 
and integrating over $(0,1)$, we have 
\begin{align*}
\dfrac{d}{dt} \left(
\io u \int_1^u a(r)\,dr 
 \right)
&= \io \dt u \int_1^u a(r)\,dr + \io u a(u)\dt u\\
&=-\io (a(u) \dx u - u \dx v )a(u) \dx u
-\io \dx (ua(u)) (a(u) \dx u - u \dx v ).
\end{align*}
Since it follows from a straightforward computation that
\begin{align*}
&-\io a(u) \dx u \cdot a(u) \dx u - \io \dx (ua(u)) a(u) \dx u
= \io \dx \left(\dfrac{a(u)}{u} \dx u \right) u^2a(u) ,\\
&\io u \dx v \cdot a(u) \dx u + \io \dx (ua(u)) u \dx v 
 = \io \dx \left( u^2 a(u) \right) \cdot \dx v,
\end{align*}
we conclude the proof. 
\end{proof}
Now we are in the position to construct the announced functional.
\begin{prop}\label{prop_new_lyapunov}
Let $(u,v)$ be a solution of \eqref{P1} in $(0,T)\times (0,1)$. 
The following identity is satisfied
\begin{eqnarray*}
\dfrac{d}{dt} \mathcal{F}(u(t)) + \mathcal{D}(u(t),v(t))
=\io \dfrac{ua(u)v^2}{4},
\end{eqnarray*}
where
\begin{eqnarray*}
\mathcal{F}(u(t)) 
&:=&\frac{1}{2} \io \dfrac{(a(u))^2}{u}|\dx u|^2
- \io u\int_1^u a(r)\,dr ,\\[2mm]
\mathcal{D}(u(t),v(t)) 
&:=& \io ua(u) 
\left|  \dx \left(\dfrac{a(u)}{u}\dx u \right) - \dx^2 v  + \dfrac{v}{2} \right|^2.
\end{eqnarray*}
\end{prop}
\begin{proof}
Multiplying the second equation of \eqref{P1} by $u a(u) \dx^2 v$ and integrating over $(0,1)$ we have that
\begin{align}\label{Est2_P1}
\notag 
0&=\io u a(u)  |\dx^2 v|^2  - \io  v \cdot u a(u) \dx^2 v  
+ \io u^2a(u) \dx^2 v\\
  &=  \io u a(u)  |\dx^2 v|^2  - \io  v \cdot u a(u) \dx^2 v  
  - \io \dx \left( u^2 a(u) \right) \cdot \dx v.
\end{align} 
Combining Lemma \ref{Est1_P1} and \eqref{Est2_P1} we get
\begin{align*}
\frac{d}{dt} \left(
- \io u\int_1^u a(r)\,dr 
 \right)
+\io ua(u) |\dx^2 v|^2  - \io  v \cdot ua(u) \dx^2 v
=- \io \dx \left( \frac{a(u)}{u} \dx u \right) \cdot u^2 a(u),
\end{align*}
and then using the second equation of \eqref{P1} we see that
\begin{align*}
&\frac{d}{dt} \left(
- \io u\int_1^u a(r)\,dr 
 \right)
+\io ua(u) |\dx^2 v|^2  - \io  v \cdot ua(u) \dx^2 v\\
&=- \io ua(u) \dx \left( \frac{a(u)}{u} \dx u \right)(- \dx^2 v + v).
\end{align*}
Thus it follows from Lemma \ref{lemma_Lyap1} that
\begin{align*}
&\dfrac{d}{dt} \left(
\frac{1}{2} \io \dfrac{(a(u))^2}{u}|\dx u|^2
- \io u\int_1^u a(r)\,dr
\right)
+ \io ua(u)
\left| \dx \left(\dfrac{a(u)}{u}\dx u \right) - \dx^2 v \right|^2\\
& + \io ua(u) v \cdot \left( \dx \left(\dfrac{a(u)}{u}\dx u \right) - \dx^2 v \right) =0,
\end{align*}
that is,
\begin{align*}
\dfrac{d}{dt} \left(
\frac{1}{2} \io \dfrac{(a(u))^2}{u}|\dx u|^2
- \io u\int_1^u a(r)\,dr
\right)
+ \io ua(u) 
\left|  \dx \left(\dfrac{a(u)}{u}\dx u \right) - \dx^2 v  + \dfrac{v}{2} \right|^2
= \io \dfrac{ua(u)v^2}{4}.
\end{align*}
which is the desired inequality.
\end{proof}
From now on, we focus on the critical cases $a(u)= 1/u$ or $a(u)=1/(1+u)$.
\begin{cor}\label{cor_lyapunov}
Let $(u,v)$ be a solution of \eqref{P1} with either $a(u)=\frac{1}{u}$ and
$u_0$ satisfying \eqref{nieujemne} or $a(u)=1/(1+u)$ and
$u_0$ satisfying \eqref{dodatnie}. 
Then the following identities are satisfied
\begin{eqnarray}
\dfrac{d}{dt} \mathcal{F}(u(t)) + \mathcal{D}(u(t),v(t))
=\io b(u)\dfrac{v^2}{4}\;,
\end{eqnarray}
where $b(u)\equiv 1$ for $a(u)=1/u$ and $b(u)=\frac{u}{1+u}$ if $a(u)=1/(1+u)$,
while $\mathcal{F}$ is given by an explicite formula as 
\begin{eqnarray*}
\mathcal{F}(u(t)) 
&:=&\frac{1}{2} \io \dfrac{|\dx u|^2}{u^3} 
- \io u\log u\;\mbox{when}\;\;a(u)=1/u,\\[2mm]
\mathcal{F}(u(t)) 
&:=&\frac{1}{2} \io \dfrac{|\dx u|^2}{u(1+u)^2}
- \io u\log (1+u)\;\mbox{if}\;\;a(u)=1/(1+u).
\end{eqnarray*}
\end{cor}
%
%
%
%
%
%
%
%
%
\section{Proof of Main theorem}\label{section_proof}
By Corollary \ref{cor_lyapunov} and Lemma \ref{prop_ellipticReg} we can see the upper bound of the functional $\mathcal{F}(u)$. 
To derive regularity estimates, we first establish the lower bound of the functional. 
The following proposition is essentially similar as \cite[Proposition 7]{CL_DCDS}. 
\begin{prop}\label{propo}
Let $(u,v)$ be a solution of \eqref{P1} with $a(u)=1/u$ 
in $(0,T)\times (0,1)$ and $u_0$ satisfying \eqref{nieujemne} and $\io u_0 = M>0$. 
The following estimates hold
\begin{align}\label{RegEst1}
\mathcal{F}(u) \geq 
\dfrac{1}{4} \io \dfrac{|\dx u|^2}{u^3} - M \log M - M^3,
\end{align}
and
\begin{align}\label{RegEst2}
\io u|\log u| 
\leq 
2 + M \log M + M^{\frac{3}{2}} \left( \io \dfrac{|\dx u|^2}{u^3} \right)^{\frac{1}{2}}.
\end{align}
\end{prop}
\begin{proof}
Let $t \in (0,T)$ be fixed. 
Since $\io u(t,x)\,dx = M$, we can find some point 
$x_0 \in [0,1]$ such that $u(t, x_0) = M$. Then
\begin{align*}
- \io u(x) \log u(x) &= \io u(x) \log \dfrac{1}{u(x)}\,dx\\
&= \io u(x) \left(\log \dfrac{1}{u(x)} - \log \dfrac{1}{u(x_0)} \right)\,dx + \io u(x)\log \dfrac{1}{u(x_0)}\,dx\\
&= \io u(x) \left(\int_{x_0}^x \dx \left( \log \dfrac{1}{u(z)} \right) \,dz \right)\,dx
- M \log M.
\end{align*}
By the Cauchy--Schwarz inequality it follows that
\begin{align}\label{RegEst3}
\notag
- \io u(x) \log u(x) & \geq
-  \io u(x) \left(\int_{x_0}^x \frac{1}{u(z)} \cdot \dfrac{|\dx u(z)|^2}{u^2(z)} \,dz \right)^{\frac{1}{2}}
\left(\int_{x_0}^x u(z) \,dz \right)^{\frac{1}{2}}\,dx
- M \log M\\
& \geq - M^{\frac{3}{2}} \left( \io \dfrac{|\dx u|^2}{u^3} \right)^{\frac{1}{2}} 
- M \log M.
\end{align}
From the above inequality we derive the lower bound for $\mathcal{F}(u)$ such that
\begin{align*}
\mathcal{F}(u) &\geq
\frac{1}{4} \io \dfrac{|\dx u|^2}{u^3} + \frac{1}{4} \io \dfrac{|\dx u|^2}{u^3} 
- M^{\frac{3}{2}} \left( \io \dfrac{|\dx u|^2}{u^3} \right)^{\frac{1}{2}} 
+ M^3 - M^3 - M \log M\\
& = \frac{1}{4} \io \dfrac{|\dx u|^2}{u^3}
+ \dfrac{1}{4} \left(\left( \io \dfrac{|\dx u|^2}{u^3} \right)^{\frac{1}{2}} - 2M^{\frac{3}{2}} \right)^2 
- M^3 - M\log M,
\end{align*}
which implies \eqref{RegEst1}. 
Since 
\begin{align*}
\left|\log \frac{1}{u}\right| = 2 \max \left\{0, \log \frac{1}{u} \right\} - \log \frac{1}{u},
\end{align*}
we see that 
\begin{align*}
\io u \left|\log \frac{1}{u}\right| 
&=2 \io u \max \left\{0, \log \frac{1}{u} \right\} - \io u\log \frac{1}{u}.
\end{align*}
It follows from \eqref{RegEst3} that
\begin{align*}
\io u \left|\log \frac{1}{u}\right| 
& \leq
2 \io u \left(\exp\left( \max \left\{0, \log \frac{1}{u} \right\} \right) -1 \right)
+M \log M
+ M^{\frac{3}{2}} \left( \io \dfrac{|\dx u|^2}{u^3} \right)^{\frac{1}{2}}\\
&\leq 2 \io u \exp\left(\log \frac{1}{u} \right) +M \log M
+ M^{\frac{3}{2}} \left( \io \dfrac{|\dx u|^2}{u^3} \right)^{\frac{1}{2}},
\end{align*}
which concludes \eqref{RegEst2}.
\end{proof}
\begin{prop}
Let $(u,v)$ be a solution of \eqref{P1} with $a(u)=1/(1+u)$ 
in $(0,T)\times (0,1)$, again $\io u_0 = M>0$. 
The following estimates hold
\begin{align}\label{RegEst1.5}
\mathcal{F}(u) \geq 
\dfrac{1}{4} \io \dfrac{|\dx u|^2}{u(1+u)^2} - M \log (1+M) - M^3,
\end{align}
and
\begin{align}\label{RegEst2.5}
\io u\log (1+u)
\leq 
M^3 + M \log (1+M) + 1/4 \left( \io \dfrac{|\dx u|^2}{u(1+u)^2} \right)^{\frac{1}{2}}.
\end{align}
\end{prop}
\begin{proof}
Exactly the same way as in \eqref{RegEst3} we obtain the estimate
\[
-\io u\log(1+u)\geq -M^{\frac{3}{2}}\left(\io \dfrac{|\dx u|^2}{u(1+u)^2}\right)^{\frac{1}{2}}-M\log (1+M),
\]
which gives as in turn (again proceeding along the lines of the proof of Proposition \ref{propo})
\eqref{RegEst1.5}. Consequently
\[
\int_0^1 u\log(1+u)dx \leq \left( 1/2-1/4 \right)\left(\io \dfrac{|\dx u|^2}{u(1+u)^2}\right)+M\log (1+M)+M^3.
\] 
\end{proof}
Here we obtain regularity estimates, which depend on the time interval $T>0$.
\begin{prop}\label{prop_regest}
Let $(u,v)$ be a solution of \eqref{P1} with $a(u)=1/u$ 
in $(0,1)\times (0,T)$. Then  there exists some constant $C>0$ such that
\begin{align*}
\io u |\log u| + \io \frac{|\dx u|^2}{u^3} \leq C(1+T)
\qquad\mbox{for all }t \in (0,T).
\end{align*}
\end{prop}
\begin{proof}
Since Lemma \ref{prop_ellipticReg} yields that there exists some constant $C>0$ such that
\begin{align*}
\io \dfrac{v^2}{4} \leq C,
\end{align*}
it follows that for all $t \in (0,T)$,
\begin{align*}
\mathcal{F}(u(t)) \leq \mathcal{F}(u_0) + CT.
\end{align*}
Thus \eqref{RegEst1} implies that
\begin{align*}
\dfrac{1}{4} \io \dfrac{|\dx u|^2}{u^3}
\leq \mathcal{F}(u(t)) + M \log M + M^3
\leq \mathcal{F}(u_0) + M \log M + M^3+ CT.
\end{align*}
Moreover combining the above inequality and \eqref{RegEst2} yields the desired inequality.
\end{proof}
\begin{remark}
For $a(u)=1/(1+u)$ in a similar way as in the previous proposition, basing on \eqref{RegEst1.5} and \eqref{RegEst2.5}, we obtain
\begin{equation}\label{raz}
\io u\log(1+u) + \io \dfrac{|\dx u|^2}{u(1+u)^2}\leq C(1+T).
\end{equation}
\end{remark}
Next we proceed with the usual regularity argument (see \cite{NSY, BCM-R}).  
\begin{lemma}
Let $(u,v)$ be a solution of \eqref{P1} with respectively $a(u)=1/u$ in $(0,1)\times (0,T)$, $u_0$ satisfying \eqref{nieujemne} or $a(u)=1/(1+u)$ and $u_0$ satisfying \eqref{dodatnie}. Then  there exists some constant $C(T)>0$ such that 
\begin{align*}
\io u^3 \leq C(T)
\qquad\mbox{for all }t \in (0,T).
\end{align*}
\end{lemma}
\begin{proof}
We focus first on the case $a(u)=1/u$. Multiplying the first equation of \eqref{P1} by $u^2$ we have that
\begin{align*}
\dfrac{1}{3}\dfrac{d}{dt} \io u^3 
&=
-2 \io |\dx u|^2 + 2 \io u^2 \dx u \dx v\\
&=
-2 \io |\dx u|^2 - \frac{2}{3}\io u^3 \dx^2 v.
\end{align*}
Using the second equation of \eqref{P1} we yield that
\begin{align*}
\dfrac{1}{3}\dfrac{d}{dt} \io u^3 
+2 \io |\dx u|^2
\leq \frac{2}{3}\io u^4.
\end{align*}
It follows form Lemma \ref{BHN_inequality} and Proposition \ref{prop_regest} that for any $\delta>0$ there exists $C_\delta>0$ such that
\begin{align*}
\dfrac{1}{3}\dfrac{d}{dt} \io u^3 
+2 \io |\dx u|^2
\leq 
\delta(1+T) \|u\|^2_{H^1(0,1)} + C_\delta \|u\|_{L^1(0,1)}.
\end{align*} 
Recalling the fact that the Gagliardo--Nirenberg inequality implies 
$$
\|u\|^2_{H^1(0,1)} \leq C \left(\io |\dx u|^2 + \|u\|^2_{L^1(0,1)} \right)
$$ 
with some $C>0$ and choosing $\delta = 1/C(1+T)$, we see that
\begin{align*}
\dfrac{d}{dt} \io u^3 
\leq C(T) \qquad \mbox{for all }t\in (0,T)
\end{align*} 
with some $C(T)>0$.
In the case of $a(u)=1/(1+u)$ we start by multiplying the equation by $(1+u)^2$ and exactly the same way we arrive at 
\begin{align*}
\dfrac{d}{dt} \io (1+u)^3 
\leq C(T) \qquad \mbox{for all }t\in (0,T).
\end{align*}
\end{proof}
\begin{proof}[Proof of Theorem \ref{main_theorem}]
By the iterative argument (see \cite{BCM-R, NSY}) we have for any $p \in (1,\infty)$
\begin{align*}
\|u(t)\|_{L^p(0,1)} \leq C(T) \qquad \mbox{for all }t\in(0,T)
\end{align*}
with some $C(T)>0$. Finally by the standard regularity estimates for quasilinear parabolic equation (\cite[Proposition 3]{BCM-R}) we can derive boundedness of $u$,
\begin{align*}
\|u(t)\|_{L^\infty(0,1)} \leq C(T)  \qquad \mbox{for all }t\in(0,T),
\end{align*}
which implies global existence of solutions to \eqref{P1}. 
\end{proof}
%
%
%
%

\bigskip
\bigskip
\textbf{Acknowledgments} \\
The second author wishes to thank Institute of Mathematics of the Polish Academy of Sciences, where the idea of this paper births, for financial support and the warm hospitality.
\end{document}